\documentclass[12pt]{article}

\usepackage{amsmath,amssymb,amsthm}
\usepackage{amscd}
\usepackage{graphicx}

\newtheorem{theorem}{Theorem}[section]
\newtheorem{lemma}[theorem]{Lemma}
\newtheorem{corollary}[theorem]{Corollary}
\newtheorem{proposition}[theorem]{Proposition}

\theoremstyle{definition}
\newtheorem{definition}[theorem]{Definition}

\theoremstyle{remark}
\newtheorem{remark}[theorem]{Remark}

\numberwithin{equation}{section}

\newcommand{\Kh}{\operatorname{Kh}}
\newcommand{\KhL}{\operatorname{Kh}_{\operatorname{Lee}}}
\newcommand{\HFK}{\operatorname{HFK}}
\newcommand{\Res}{\operatorname{Res}}
\newcommand{\lk}{\operatorname{lk}}

\begin{document}

\title{The link concordance invariant from Lee homology}

\author{John Pardon}

\date{23 July 2011; Revised 9 February 2012}

\maketitle

\begin{abstract}
We use the knot homology of Khovanov and Lee to construct link concordance invariants generalizing the Rasmussen 
$s$-invariant of knots.  The relevant invariant for a link is a filtration on a vector space of dimension 
$2^{\left|L\right|}$.  The basic properties of the $s$-invariant all extend to the case of links; in particular, 
any orientable cobordism $\Sigma$ between links induces a map between their corresponding vector spaces which 
is filtered of degree $\chi(\Sigma)$.  A corollary of this construction is that any component-preserving orientable 
cobordism from a $\Kh$-thin link to a link split into $k$ components must have genus at least $\lfloor\frac k2\rfloor$.  
In particular, no quasi-alternating link is concordant to a split link.
\end{abstract}


\section{Introduction}

Using Lee's modification \cite{lee} of Khovanov homology \cite{khovanov}, Rasmussen \cite{rasmussen} introduced 
for every knot $K$ an even integer valued invariant, known as the $s$-invariant.  It 
shares some of the basic properties of the classical knot signature; in particular it is a homomorphism from the group of smooth concordance 
classes of knots to $2\mathbb Z$, and gives a lower bound for twice the smooth slice genus (though the signature also 
does this in the \emph{topological} category, whereas the $s$-invariant does not).  The definition of the $s$-invariant 
is purely combinatorial, and, like many other knot invariants coming out of quantum algebra, it so far lacks any intrinsic geometric definition.  
One of the main reasons for interest in this invariant is that it is by definition algorithmically computable (though some 
cleverness is needed to do large calculations quickly, c.f.\ Bar-Natan \cite{bnalg} and Freedman, Gompf, Morrison, and Walker \cite{spc4calc}), 
and is one of the few tools known to give useful lower bounds on the smooth slice genus of knots.  In particular, 
Rasmussen \cite{rasmussen} showed via direct calculation that $s(T_{p,q})=(p-1)(q-1)$, thus proving the Milnor conjecture that 
$g_4(T_{p,q})=\frac 12(p-1)(q-1)$, a hard theorem of Kronheimer and Mrowka \cite{kmmilnord1,kmmilnord2} \cite{kmmilnorsw} 
proved (twice) using gauge theory.

In this paper, we consider the natural generalization of the $s$-invariant to a concordance invariant of links.  
Everything we do will be in the smooth category.  Since the $s$-invariant 
can detect the deep differences between the smooth and topological categories in four dimensions, this restriction is 
in fact necessary for this theory.  In particular, knot and link concordance is meant in the smooth sense.

Let us denote the Khovanov--Lee homology groups of a link by $\KhL^\ast(L)$.  Lee \cite{lee} showed that 
$\KhL^\ast(L)$ is a surprisingly simple group: there is an isomorphism 
$\bigoplus_{\text{orientations of }L}\mathbb Q\xrightarrow{\sim}\KhL^\ast(L)$.  
We denote the former group by $\mathbb O(L)$, and in Section \ref{orientationgroupsec}, we will 
define it as a functor (we will specify the maps associated to cobordisms).  Kevin Walker 
\cite{kwalker} informs the author that (with suitable choice of Lee deformation parameter) there is an equivalence of functors 
between $\mathbb O$ and $\KhL^\ast$ (Rasmussen \cite{rasmussen} \cite{rasrefined} has proved a sort of approximate 
equivalence of functors).  The natural generalization of the $s$-invariant is thus the pull-back of 
the $s$-filtration on $\KhL^\ast(L)$ to a filtration on $\mathbb O(L)$.  To get a numerical invariant, 
we can take the following (which is perhaps slightly coarser).

\newcommand{\definitionofinvariant}{For an oriented link $L\subseteq\mathbb R^3$, we associate 
a function $d_L:\mathbb Z\times\mathbb Z\to\mathbb Z_{\geq 0}$ so that $d_L(h,s)$ gives the 
dimensions of the associated graded pieces $(\KhL^h(L))^s/(\KhL^h(L))^{s+1}$ (where $(\KhL^\ast(L))^s$ 
denotes the subspace of elements of filtration level $\geq s$).}

\begin{definition}
\definitionofinvariant
\end{definition}

For a knot $K$, it is a theorem of Rasmussen \cite{rasmussen} that $d_K(0,s(K)\pm 1)=1$ and $d_K$ is otherwise zero 
(this being the defining property of the $s$-invariant).  For a link $L$, the vector space $\KhL^\ast(L)$ has dimension 
$2^{\left|L\right|}$, and, as one might expect, the support of the function $d_L$ can be much more complicated as we 
shall see in a few examples.

\begin{theorem}\label{mainthm2}
Let $L$ be a link with orientation $\mathfrak o$.  The invariant $d_L:\mathbb Z\times\mathbb Z\to\mathbb Z_{\geq 0}$ satisfies the following basic properties:
\begin{enumerate}
\item\label{supportmod}$\sum_{s\equiv\left|L\right|+k\mod 4}d_L(h,s)$ is zero if $k$ is odd, and if $k$ is even it 
equals one half the number of orientations $\mathfrak o_1$ of $L$ such that $\lk(\mathfrak o_1)-\lk(\mathfrak o)=-h$.  
Here $\lk(\mathfrak o)=\sum_{i<j}\lk(L_i^\mathfrak o,L_j^\mathfrak o)$ (sum over the components of $L$).
\item\label{tensor}$d_{L_1\sqcup L_2}=d_{L_1}\ast d_{L_2}$ (convolution).
\item$d_{\bar L}(h,s)=d_{L}(-h,-s)$.
\item\label{filtprop} If $\Sigma$ is a component-preserving orientable cobordism between $L_1$ and $L_2$ (i.e.\ $H_0(L_i)\overset\sim\to H_0(\Sigma)$), 
then $\sum_{s\geq a}d_{L_1}(h,s)\leq\sum_{s\geq a+\chi(\Sigma)}d_{L_2}(h,s)$ for all $h\in\mathbb Z$.
\item$d_L$ is a link concordance invariant.
\item\label{orchange}$d_{(L,\mathfrak o)}(h+\lk(\mathfrak o),s+3\lk(\mathfrak o))$ is independent of orientation $\mathfrak o$.
\end{enumerate}
\end{theorem}

For links with a large number of components, it is reasonable to expect that the invariant $d_L$ will be a strong invariant of 
link concordance.  As one sees in the theorem above, the invariant $d_L$ is best suited for studying cobordisms 
which do not merge components of $L$.  In general, if one wants to derive information about a given orientable cobordism, 
then the relevant object is the $s$-filtration restricted to the subspace of $\mathbb O(L)$ generated by those orientations 
extending to orientations of the cobordism.  The larger this subspace, the more likely the invariant is to be useful.

Beliakova and Wehrli have defined an integer $s(L,\mathfrak o)$ for a link with an orientation \cite{wehrli}.  This corresponds 
to the $s$-filtration restricted to the $2$-dimensional subspace of $\KhL^\ast(L)$ generated by that orientation and its reverse.  
Just like for knots, one shows that on this subspace, the filtration is supported in two levels $s\pm 1$, and this defines $s(L,\mathfrak o)$.  
This invariant is best suited for studying oriented cobordisms which are allowed to merge components of $L$.  Examples 
show that the function $\mathfrak o\mapsto s(L,\mathfrak o)$ is a weaker invariant that $d_L$.  One expects 
that $d_L$ is  a weaker invariant that the filtration on $\mathbb O(L)$ but we don't have any examples to prove this 
at present (mainly because $d_L$ is often easy to derive from $\Kh^\ast$---which there exist programs to 
compute---whereas the $s$-filtration on $\mathbb O(L)$ is not).  We discuss examples in Section \ref{examples} and at 
the end of Section \ref{properties}.

In Section \ref{applications}, we use the invariants $d_L$ to derive the following corollary, which appears to be new.

\begin{corollary}\label{maincor}
A component-preserving orientable cobordism between a $\Kh$-thin link and a link split into $m$ components must have genus at least $\lfloor\frac m2\rfloor$.  
In particular, $\Kh$-thin links (in particular quasi-alternating links [Definition \ref{quasialterdef}]) are not concordant to split links.
\end{corollary}

It is known (via properties of the Alexander module) that alternating links are not concordant to split links \cite{alternation}.  
It would be interesting to try to prove Corollary \ref{maincor} (say, restricted to quasi-alternating links) using the Alexander module.\footnote{We recently learned of a preprint \cite{powell} by Stefan Friedl and Mark Powell which apparently presents such a proof.}

This corollary is interesting because the $s$-invariant for alternating \emph{knots} is equal to the knot signature, and thus 
gives no new information (the inequality $g_4^{\operatorname{top}}(K)\geq\frac 12\left|\sigma(K)\right|$ is classical, see 
Murasugi \cite[p416, Theorem 9.1]{signaturemurasugi}).  It is interesting to note that Khovanov homology has a 
reputation for being easy to compute (at least, compared to gauge theoretic invariants which give results similar to 
the Milnor conjecture), but hard to use to prove general theorems, since its structure in general is still poorly understood.  
Thus the above corollary is interesting in that it is a \emph{general} statement which doesn't intrinsically involve 
Khovanov homology (at least, if one restricts to quasi-alternating links).

There have recently been efforts (see Freedman, Gompf, Morrison, and Walker \cite{spc4calc}) to prove that some 
specific proposed counterexamples to the smooth $4$-dimensional Poincar\'e conjecture are in fact exotic by proving 
some specific links are not slice in the standard $\mathbb B^4$ (links which, by virtue of coming from Kirby diagrams for the proposed counterexample, 
are by definition slice in the proposed exotic $\mathbb B^4$).  By slice, we mean \emph{strongly slice}, i.e.\ bounding a 
disjoint union of disks in $\mathbb B^4$.  Since these are usually multi-component links, it may be 
helpful to compute the entire filtration on $\KhL^\ast(L)$: for a link with many components, this \emph{a priori} may 
be a much stronger invariant than the set of $s$-invariant values for some associated knots which are implied to be slice if the link is slice 
(computing these $s$-invariant values was the strategy employed in \cite{spc4calc}).  
We should, however, also note that, in accordance with the growing relations between Khovanov homology and gauge theory, 
some would conjecture that the $s$-filtration should be invariant under concordance of links in any \emph{homotopy} 
$\mathbb R^3\times[0,1]$, and thus would not imply in any straightforward manner that any homotopy $\mathbb B^4$ is exotic.

One thinks that an invariant of links similar to $d_L$ could be defined using the Link Floer Homology of Ozsvath--Szab\'o 
\cite{oszhfk,oszhfl} as an appropriate generalization of the $\tau$-invariant.  One would expect this invariant to satisfy similar properties as the 
$s$-filtration on $\KhL^\ast(L)$.  It is perhaps interesting to note that the vector space $\mathbb O(L)$ appears 
in the Link Floer Homology theory in the guise of $\wedge^\ast H_1(\#^{\left|L\right|}\mathbb S^1\times\mathbb S^2)$ 
(once we take the union of our link with the unknot).

\subsection{Acknowledgments}

This paper represents part of the author's Senior Thesis at Princeton University, 
advised by Zolt\'an Szab\'o.  I thank him for lots of generous time spent meeting and 
discussing mathematics.  John Baldwin, the second reader for my thesis, also made some 
useful comments.  I thank the referee for a very close reading of this paper and 
for the many resulting corrections and clarifications, as well as contributions 
concerning the examples in Section \ref{examples}.

\section{Khovanov--Lee homology}

In this section, we give a quick review of Lee's deformation \cite{lee} of Khovanov homology \cite{khovanov} 
aimed at our intended application.  For a good introduction to Khovanov homology, see Bar-Natan's articles 
\cite{barnatanexpo} and \cite{barnatan}.  The maps for cobordisms were first proved consistent by Jacobsson \cite{jacob}.

To be completely explicit, we define Khovanov--Lee homology via Khovanov's chain complex using the following Frobenius algebra $V$:
\begin{align}
V&=\mathbb Q\mathbf v_-\oplus\mathbb Q\mathbf v_+&\iota(1)&=\mathbf v_+\cr
\epsilon(\mathbf v_+)&=0&m(\mathbf v_-\otimes\mathbf v_-)&=a\mathbf v_+\cr
\epsilon(\mathbf v_-)&=1&m(\mathbf v_-\otimes\mathbf v_+)&=\mathbf v_-\cr
\Delta(\mathbf v_+)&=\mathbf v_-\otimes\mathbf v_++\mathbf v_-\otimes\mathbf v_+&m(\mathbf v_+\otimes\mathbf v_-)&=\mathbf v_-\cr
\Delta(\mathbf v_-)&=\mathbf v_-\otimes\mathbf v_-+a\mathbf v_+\otimes\mathbf v_+&m(\mathbf v_+\otimes\mathbf v_+)&=\mathbf v_+
\end{align}
Setting $a=0$ yields Khovanov homology ($a$ is the Lee deformation parameter).  If $a\ne 0$, then $(m,\iota,\Delta,\epsilon)$ 
admit simple descriptions in terms of the basis $\mathbf x_\pm=\mathbf v_-\pm\sqrt a\mathbf v_+$, and this implies that the 
resulting homology is essentially isomorphic to Lee homology.  The only difference between different values of $a\ne 0$ is 
that the maps associated to a cobordism $\Sigma$ carry a factor of $(2\sqrt a)^{-\chi(\Sigma)/2}$.  This makes Lee's original 
choice of $a=1$ slightly inconvenient, so for the remainder of the paper we set $a=\frac 14$ (as suggested by Walker \cite{kwalker}).

\begin{theorem}\label{khleedef}
For every oriented link $L$, there is an associated $\mathbb Z$-graded vector space $\KhL^\ast(L)$ over $\mathbb Q$ (the grading $\ast$ 
is called the \emph{homological} grading).  Furthermore, each $\KhL^h(L)$ carries a descending filtration, called the $s$-filtration.  
Every oriented cobordism $\Sigma\subseteq\mathbb R^3\times[0,1]$ from $L_1$ to $L_2$ induces a homomorphism $F_\Sigma:\KhL^\ast(L_1)\to\KhL^\ast(L_2)$ 
(defined up to $\pm 1$) which respects the homological grading, and which is filtered of degree $\chi(\Sigma)$.  We have the following additional properties:
\begin{enumerate}
\item\label{smodfour}$\KhL^h(L)$ carries an absolute $\mathbb Z/4\mathbb Z$ grading which is supported in gradings $\equiv\left|L\right|\mod 2$, 
and these two pieces have equal dimensions.  The $s$-filtration breaks up as a filtration on each of the pieces, and 
the $s$-filtration on the degree $k\in\mathbb Z/4\mathbb Z$ piece is supported on integers $s\equiv k\mod 4$.
\item\label{tensorproperty}$\KhL^\ast(L_1\sqcup L_2)=\KhL^\ast(L_1)\otimes\KhL^\ast(L_2)$ (naturally), and this is an isomorphism of the homological grading 
and the $s$-filtration.
\item$\KhL^\ast(\bar L)$ is naturally the dual of $\KhL^\ast(L)$.
\item(due to Lee \cite{lee}) $\dim\KhL^\ast(L)=2^{\left|L\right|}$.  In fact, $\dim\KhL^h(L)$ is the number of 
orientations $\mathfrak o$ of $L$ such that $\lk(\mathfrak o)-\lk(\mathfrak o_1)=-h$, where $\mathfrak o_1$ is the given 
orientation of $L$, and $\lk(\mathfrak o)=\sum_{i<j}\lk(L_i^\mathfrak o,L_j^\mathfrak o)$ (sum over the components of $L$).
\item$\KhL^\ast$ is a functor from the appropriately defined category of links and cobordisms (see \cite{fixingfunctoriality}).
\item$\KhL^{\ast+\lk(\mathfrak o)}(L,\mathfrak o)_{(-3\lk(\mathfrak o))}$ is independent of orientation $\mathfrak o$ 
(where $_{(q_0)}$ means an upwards shift of the $s$-filtration by $q_0$).
\end{enumerate}
\end{theorem}

\begin{remark}
Clark, Morrison, and Walker \cite{fixingfunctoriality} and Caprau \cite{caprau} have shown how to define Khovanov--Lee homology 
(with indeterminate $a$) so that the maps associated to cobordisms no longer have a sign ambiguity.  This requires adjoining $i=\sqrt{-1}$ to 
the coefficient ring.
\end{remark}

\begin{definition}
\definitionofinvariant
\end{definition}

Theorem \ref{mainthm2} (the basic properties of $d_L$) follows directly from the basic properties of $\KhL^\ast$ listed 
in Theorem \ref{khleedef}.

\section{Applications to link concordance}\label{applications}

\begin{definition}
A cobordism $\Sigma$ between two links $L_1$ and $L_2$ is said to be \emph{component-preserving} iff $H_0(L_1)\xrightarrow\sim H_0(\Sigma)\xleftarrow\sim H_0(L_2)$.  Note that a component-preserving orientable cobordism of genus $0$ is exactly a link concordance.
\end{definition}

\begin{remark}
One is perhaps also interested in relaxing the restrictive notion of \emph{component}-preserving cobordism to \emph{color}-preserving 
cobordism, where multiple components of the link could have the same color.  Now certainly this case is also easily handled 
using the invariant $\KhL^\ast(L)$.  The necessary data is a coloring of the link, and a choice of relative orientation on 
each colored component (by relative orientation, we mean an orientation up to overall reversal).  Then the relevant invariant is 
just the restriction of the $s$-filtration to the subspace of $\KhL^\ast(L)$ 
generated by all orientations agreeing with the given relative orientations on each colored component.
\end{remark}

\begin{lemma}\label{isomlemma}
The map $F_\Sigma:\KhL^\ast(L_1)\to\KhL^\ast(L_2)$ induced by a component-preserving orientable cobordism is an isomorphism of vector spaces.
\end{lemma}

\begin{proof}
This follows from Rasmussen \cite[p434 Proposition 4.1]{rasmussen}.
\end{proof}

\begin{figure}
\centering\includegraphics{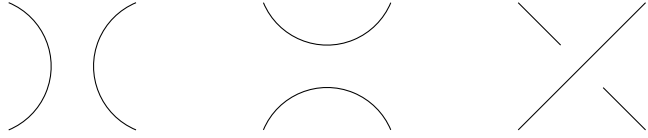}
\caption{Local pictures of $L_0$, $L_1$, $L_\infty$.}\label{L01inftyfig}
\end{figure}

\begin{definition}
A link $L$ is said to be $\Kh$-thin iff $\Kh^\ast(L)$ is supported on exactly two diagonals of the form $q=q_0+2h\pm 1$ ($h=\ast$ is 
the homological grading).
\end{definition}

The spectral sequence from $\Kh^\ast(L)$ to $\KhL^\ast(L)$ implies that for a $\Kh$-thin link, 
the support of $d_L$ is contained in the same two diagonals $s=q_0+2h\pm 1$.

\begin{definition}\label{quasialterdef}
Oszv\'ath and Szab\'o \cite{oszquasialter} define the set of \emph{quasi-alternating} links to be the set of links generated by the unknot 
using the following skein operation: if $L_0$ and $L_1$ are quasi-alternating and $\det L_\infty=\det L_0+\det L_1$, then 
$L_\infty$ is also quasi-alternating (where $L_0,L_1,L_\infty$ are given as in Figure \ref{L01inftyfig}).
\end{definition}

It is standard that all non-split alternating links are quasi-alternating.  
Quasi-alternating links are known to be both $\Kh$-thin and $\widehat{\HFK}$-thin by Manolescu--Ozsv{\'a}th \cite{quasithin}, 
though Greene \cite{greene} has shown that there are non-quasi-alternating links that are both $\Kh$-thin and $\widehat{\HFK}$-thin.

\begin{proposition}[Corollary \ref{maincor}]\label{mainresult}
Let $L$ be a $\Kh$-thin link, and suppose $\Sigma$ is a component-preserving orientable cobordism between $L$ 
and $M=M_1\sqcup\cdots\sqcup M_k$.  Then $g(\Sigma)\geq\lfloor\frac k2\rfloor$.
\end{proposition}

\begin{proof}
Fix an orientation on $L$, which thus orients each $M_i$.

We know (from Theorem \ref{mainthm2} property 
\ref{supportmod}) that the support of the $s$-filtration on $\KhL^0(M_i)$ has diameter at least $2$.  
Thus (by Theorem \ref{mainthm2} property \ref{tensor}) $\KhL^0(M)$ has $s$-filtration of diameter at 
least $2k$.  Since $L$ is $\Kh$-thin, the $s$-filtration on $\KhL^0(L)$ has diameter 
equal to $2$ (using the spectral sequence from $\Kh^\ast$ to $\KhL^\ast$).

The cobordism and its reverse induce two maps:
\begin{equation}
\KhL^0(L)\xrightarrow{F_\Sigma}\KhL^0(M)\xrightarrow{F_{-\Sigma}}\KhL^0(L)
\end{equation}
These are both isomorphisms by Lemma \ref{isomlemma}.  Also, we know that both maps are filtered of degree $-2g(\Sigma)$.

Without loss of generality, suppose the $s$-filtration on $\KhL^0(L)$ is supported in degrees $\pm 1$.  Then since the isomorphism 
$\KhL^0(M)\xrightarrow{F_{-\Sigma}}\KhL^0(L)$ is filtered of degree $-2g(\Sigma)$, 
the $s$-filtration on $\KhL^0(M)$ must be supported in degrees $\leq 1+2g(\Sigma)$.  Similarly, looking at $\KhL^0(L)\xrightarrow{F_\Sigma}\KhL^0(M)$, 
we see that the $s$-filtration on $\KhL^0(M)$ must be supported in degrees $\geq -1-2g(\Sigma)$.  Thus we have 
$2+4g(\Sigma)\geq 2k$, so $g(\Sigma)\geq\lceil\frac{k-1}2\rceil=\lfloor\frac k2\rfloor$.
\end{proof}

The following corollary to Proposition \ref{mainresult} is already known via properties of the Alexander module \cite{alternation}.

\begin{corollary}
No non-split alternating link is concordant to a split link.
\end{corollary}

\section{The orientation group}\label{orientationgroupsec}

In this section we define a (almost tautological) $(1+1)$-dimensional (projective) TQFT which we call the \emph{orientation group}.  It 
is isomorphic to the TQFT used to define Lee homology (with Lee deformation parameter $a=\frac 14$).  In fact, Kevin Walker 
\cite{kwalker} informs us that the orientation group is isomorphic to $\KhL^\ast$ as a functor.  
The goal of the construction in this section is to give a natural intrinsic description of the maps associated to cobordisms.

For any manifold $X$, we let $\left|X\right|$ denote the number of connected components of $X$.

\begin{definition}
For an orientable manifold $X$, let $O(X)$ denote the set of orientations of $X$.  
Let $\mathbb O(X)$ denote the $\mathbb Q$-vector space with basis indexed by $O(X)$.  
We also define a natural inner product $\langle\cdot,\cdot\rangle$ on $\mathbb O(X)$ by declaring that this basis be orthonormal.
\end{definition}

\begin{definition}
Let $\mathfrak o\mapsto\bar{\mathfrak o}$ denote reversal of orientation; this is an involution of $O(X)$ and of $\mathbb O(X)$.
\end{definition}

By a \emph{relative orientation} on a manifold $X$, we mean an orientation up to overall reversal of orientation, that is, an element of $O(X)/(\mathfrak o\mapsto\bar{\mathfrak o})$ (which we often think of as a pair $(\mathfrak o,\bar{\mathfrak o})$).

\begin{definition}
We define a mod $4$ grading on $\mathbb O(X)$ by declaring that the $+1$ eigenspace of 
$\mathfrak o\mapsto\bar{\mathfrak o}$ have grading $-\left|X\right|$ and that 
the $-1$ eigenspace of $\mathfrak o\mapsto\bar{\mathfrak o}$ have grading $2-\left|X\right|$.
\end{definition}

\begin{lemma}
We have a natural isomorphism $\mathbb O(X_1\cup X_2)=\mathbb O(X_1)\otimes\mathbb O(X_2)$ which respects the involution $\mathfrak o\mapsto\bar{\mathfrak o}$ as well as the mod $4$ grading.
\end{lemma}

\begin{proof}
Clearly $O(X_1\cup X_2)=O(X_1)\times O(X_2)$, and this gives us the desired isomorphism of vector spaces, which clearly respects 
reversal of orientation.  Now by examining the definition of the mod $4$ grading in terms of the map $\mathfrak o\mapsto\bar{\mathfrak o}$, one 
easily sees that this implies that the mod $4$ grading is preserved as well.
\end{proof}

Henceforth we shall only be interested in $\mathbb O(X)$ in the case that $X$ is a $1$-manifold.

\begin{definition}\label{functdef}
If $A$ is an orientable cobordism between $X$ and $Y$, then we define a map $F_A:\mathbb O(X)\to\mathbb O(Y)$ (up to overall multiplication 
by $\pm 1$) as follows.  Let $\sigma_A:O(A)\to\{\pm 1\}$ satisfy the property that reversing the orientation on some component 
$A_1\subseteq A$ multiplies the value of $\sigma_A$ by $(-1)^{(\chi(A_1)-\left|A_1\cap X\right|+\left|A_1\cap Y\right|)/2}$ (note that 
since $A$ is orientable, $\chi(A_1)-\left|A_1\cap X\right|+\left|A_1\cap Y\right|\equiv\chi($closed surface$)\equiv 0\mod 2$).  
Clearly there are two such functions $\sigma_A$, differing by a sign.  Then we define (up to $\pm 1$):
\begin{equation}\label{functdefeqn}
F_A(\alpha):=\sum_{\mathfrak o\in O(A)}\sigma_A(\mathfrak o)\langle\alpha,\left.\mathfrak o\right|_X\rangle\left.\mathfrak o\right|_Y
\end{equation}
\end{definition}

By definition, orientations of $X$ which do not extend to $A$ get annihilated by $F_A$.  More generally, an orientation is sent 
to a linear combination of those orientations on $Y$ which are compatible with the cobordism $A$ and the input orientation of $X$.  
Rasmussen \cite[p434 Proposition 4.1]{rasmussen} showed a similar property of $\KhL^\ast$ in the process of defining the $s$-invariant.

\begin{lemma}\label{functlemma}
The maps associated to cobordisms are functorial in the sense that if $A$ is a cobordism between $X$ and $Y$ and $B$ is 
a cobordism between $Y$ and $Z$, then $F_{A\cup_YB}=F_B\circ F_A$.
\end{lemma}

\begin{proof}
\begin{align}\label{functcalc}
F_B(F_A(\alpha))&=\sum_{\mathfrak o_B\in O(B)}\sum_{\mathfrak o_A\in O(A)}\sigma_B(\mathfrak o_B)\sigma_A(\mathfrak o_A)\langle\alpha,\left.\mathfrak o_A\right|_X\rangle\langle\left.\mathfrak o_A\right|_Y,\left.\mathfrak o_B\right|_Y\rangle\left.\mathfrak o_B\right|_Z\cr
&=\sum_{\mathfrak o\in O(A\cup_YB)}\sigma_A(\left.\mathfrak o\right|_A)\sigma_B(\left.\mathfrak o\right|_B)\langle\alpha,\left.\mathfrak o\right|_X\rangle\left.\mathfrak o\right|_Z
\end{align}
Now just observe that the function $O(A\cup_YB)\to\{\pm 1\}$ given by $\sigma_A(\left.\mathfrak o\right|_A)\sigma_B(\left.\mathfrak o\right|_B)$ satisfies the 
property which defines $\sigma_{A\cup_YB}:O(A\cup_YB)\to\{\pm 1\}$ for the construction of $F_{A\cup_YB}$.
\end{proof}

\begin{lemma}
The map $F_A$ on $\mathbb O$ is homogeneous of degree $\chi(A)$ with respect to the mod $4$ grading.
\end{lemma}

\begin{proof}
Note that by definition of the mod $4$ grading, we have:
\begin{align}
\overline{F_A(\bar\alpha)}=F_A(\alpha)&\iff F_A\text{ homogeneous of degree }\left|X\right|-\left|Y\right|\cr
\overline{F_A(\bar\alpha)}=-F_A(\alpha)&\iff F_A\text{ homogeneous of degree }2+\left|X\right|-\left|Y\right|
\end{align}
Now we calculate:
\begin{align}
\overline{F_A(\bar\alpha)}&=\sum_{\mathfrak o\in O(A)}\sigma_A(\mathfrak o)\langle\bar\alpha,\left.\mathfrak o\right|_X\rangle\left.\bar{\mathfrak o}\right|_Y\cr
&=\sum_{\mathfrak o\in O(A)}\sigma_A(\bar{\mathfrak o})\langle\alpha,\left.\mathfrak o\right|_X\rangle\left.\mathfrak o\right|_Y
\end{align}
Now by the definition of $\sigma_A$, this equals $(-1)^{(\chi(A)-\left|X\right|+\left|Y\right|)/2}F_A(\alpha)$.  Thus we have:
\begin{align}
\chi(A)-\left|X\right|+\left|Y\right|\equiv 0\mod 4&\implies F_A\text{ homogeneous of degree }\left|X\right|-\left|Y\right|\cr
\chi(A)-\left|X\right|+\left|Y\right|\equiv 2\mod 4&\implies F_A\text{ homogeneous of degree }2+\left|X\right|-\left|Y\right|
\end{align}
which exactly says $F_A$ is homogeneous of degree $\chi(A)$.
\end{proof}

The following description shows the isomorphism with Lee's TQFT (with $a=\frac 14$).

\begin{lemma}\label{functelement}
The map $F_A$ has the following alternative description.  We decompose $A$ into iterated handle additions 
(e.g.\ using a Morse function on $A$), and then to each of the handle additions, we associate maps as follows.

For a $0$-handle, we map $\alpha$ to $\alpha\otimes(\mathfrak o-\bar{\mathfrak o})$, where $\mathfrak o$ is an orientation on the new circle.

For a $1$-handle which splits a component, the map sends every orientation to its extension to the new manifold.

For a $1$-handle which joins two components, the map sends orientations which do not extend to the new manifold to 
zero, and sends orientations which do extend to their natural extension \emph{multiplied by $\pm 1$} depending on 
the orientation of the new merged circle.

For a $2$-handle, the map sends $\mathfrak o\otimes\alpha$ to $\alpha$ and $\bar{\mathfrak o}\otimes\alpha$ to $\alpha$.
\end{lemma}

\begin{proof}
That it suffices to splice together the maps for elementary cobordisms follows from Lemma \ref{functlemma}.  We 
just have to calculate the maps coming from $k$-handle additions, $k\in\{0,1,2\}$.  These are given completely 
explicitly by Definition \ref{functdef}, which gives the result.
\end{proof}

It is interesting to note that even with this trivial construction, there is a good reason why if we want to make 
$\mathbb O(L)$ into a functor, we have no choice but to use maps are only defined up to $\pm 1$.  
For instance, consider the birth of a circle.  Note that the birth of a circle is the same cobordism as the birth of 
a circle followed by an isotopy from the circle to itself which reverses orientation.  Thus they must induce the same map.  
However, the image of the birth of a circle is $\mathfrak o-\bar{\mathfrak o}$, and this clearly changes sign under the isotopy.

If we are interested in links \emph{embedded in $\mathbb R^3$}, and we want functoriality with respect to orientable cobordisms 
\emph{embedded in $\mathbb R^3\times[0,1]$}, then it is probably possible to twist by an appropriate homomorphism 
$\pi_1(\{\text{unoriented loops in }\mathbb R^3\})\to\{\pm 1\}$ to get rid of the sign ambiguity in $\mathbb O$.  We note that Hatcher \cite{hatcher} 
has proved the Smale Conjecture, which is equivalent to the fact that the space of unoriented \emph{unknotted} loops in $\mathbb R^3$ deformation retracts 
onto the space of unoriented circles in $\mathbb R^3$, and the fundamental group of this space is indeed $\mathbb Z/2\mathbb Z$.  
We suspect this type of twisting is morally what fixes the functoriality of Khovanov homology as in \cite{fixingfunctoriality} and \cite{caprau}.

\subsection{Properties of the $s$-filtration on $\mathbb O(L)$}\label{properties}

Under the equivalence between $\KhL^\ast(L)$ and $\mathbb O(L)$, we get a natural definition of the $s$-filtration on $\mathbb O(L)$.  
The space $\mathbb O(L)$ carries a number of natural operations, and it is reasonable to ask how they respect the $s$-filtration.  
We answer a few of these questions in this section, using only the functorial properties of $\mathbb O(L)$ under cobordism.  
Because we use these soft methods, the properties we derive here would also be valid for a hypothetical generalization 
of the $\tau$-invariant to links.

The following is a rough analogue of Livingston's result \cite{livingston} that $s(K_-)\leq s(K_+)\leq s(K_-)+2$ 
(here $K_-$ and $K_+$ differ at exactly one crossing, which is positive for $K_+$ and negative for $K_-$).

\begin{lemma}\label{crossingchangelemma}
Suppose $L_1$ and $L_2$ differ by a single crossing change.  There is of course a natural isomorphism 
$\phi:\mathbb O(L_1)\overset\sim\to\mathbb O(L_2)$.  Let $\mathbb O(L_1)^+$ denote the space generated 
by orientations in which the given crossing is positive (and similarly define $\mathbb O(L_1)^-$, $\mathbb O(L_2)^+$, 
and $\mathbb O(L_2)^-$).  Pick a strand at the given crossing and an orientation of that strand.  Let 
$\psi:\mathbb O(L_2)\to\mathbb O(L_2)$ be defined by $\psi(\mathfrak o)=\sigma(\mathfrak o)\mathfrak o$, where $\sigma(\mathfrak o)=1$ 
if $\mathfrak o$ agrees with the chosen orientation on the chosen strand, and $\sigma(\mathfrak o)=-1$ otherwise.  Then we have:
\begin{enumerate}
\item\label{all}$\psi\circ\phi:\mathbb O(L_1)\to\mathbb O(L_2)$ is filtered of degree $-2$.
\item\label{positive}$\phi:\mathbb O(L_1)^-\to\mathbb O(L_2)^+$ is filtered of degree $0$.
\end{enumerate}
\end{lemma}

\begin{proof}
Our strategy is to find cobordisms which induce the required maps.

For statement \ref{all}, consider the following.  Passing the two strands through each other yields an immersed 
cobordism of Euler characteristic $0$ from $L_1$ to $L_2$.  There are two possible resolutions of 
the double point, giving two maps $\mathbb O(L_1)\to\mathbb O(L_2)$.  Using (\ref{functdefeqn}), 
we see that the two maps are:
\begin{align}
\label{plusmap}&\mathbb O(L_1)\xrightarrow{\text{projection}}\mathbb O(L_1)^+\xrightarrow{\left.\phi\right|_{\mathbb O(L_1)^+}}\mathbb O(L_2)\xrightarrow\psi\mathbb O(L_2)\\
\label{minusmap}&\mathbb O(L_1)\xrightarrow{\text{projection}}\mathbb O(L_1)^-\xrightarrow{\left.\phi\right|_{\mathbb O(L_1)^-}}\mathbb O(L_2)\xrightarrow\psi\mathbb O(L_2)
\end{align}
Since the Euler characteristic of each resolved cobordism is $-2$, both of these maps are filtered of degree $-2$.  
The sum of the two projections is the identity map on $\mathbb O(L_1)$, so the sum of (\ref{plusmap}) and (\ref{minusmap}) is 
just $\psi\circ\phi$; hence it is filtered of degree $-2$ as well.

For statement \ref{positive}, consider the following.  By Rasmussen, $\mathbb O(T_{2,3})$ is supported in $s$-filtration 
levels $1$ and $3$.  Since the mod $4$ grading agrees with the $s$-filtration, we see that $\mathfrak o+\bar{\mathfrak o}$ 
lies in filtration level $3$.  Thus the map $\mathbb O(L_1)^-\to\mathbb O(L_1)^-\otimes\mathbb O(T_{2,3})$ given by 
$\alpha\mapsto\alpha\otimes(\mathfrak o+\bar{\mathfrak o})$ is filtered of degree $3$.  Now consider an immersed cobordism starting 
at $L_1\sqcup T_{2,3}$ which first passes the strands of the crossing of $L_1$ through each other 
to get $L_2$, then unknots the $T_{2,3}$ in a similar manner, and then merges the resulting unknot with $L_2$.  The two double points 
are of opposite signs (when the crossing goes from negative in $L_1$ to positive in $L_2$), so they can be tubed 
together to obtain a cobordism of genus $1$.  Thus the resulting map $\mathbb O(L_1)^-\otimes\mathbb O(T_{2,3})\to\mathbb O(L_2)^+$ 
is filtered of degree $-3$.  The composite is $\phi:\mathbb O(L_1)^-\to\mathbb O(L_2)^+$ (as is clear from (\ref{functdefeqn})), so we are done.
\end{proof}

\begin{definition}
Given a specific orientation $\mathfrak o_i$ on a component $L_i$ of $L$, let $\Res_{\mathfrak o_i}:\mathbb O(L)\to\mathbb O(L)$ 
be orthogonal projection onto the subspace where $L_i$ is oriented by $\mathfrak o_i$.  
For a relative orientation $\mathfrak o_{ij}$ of $L_i\cup L_j$ (two components of $L$), let 
$\Res_{\mathfrak o_{ij},\bar{\mathfrak o}_{ij}}:\mathbb O(L)\to\mathbb O(L)$ 
be projection onto the subspace where $L_i\cup L_j$ has this relative orientation, composed with multiplication 
by $\sigma:O(L)\to\{\pm 1\}$ which flips sign depending on the orientation on $L_i$.
\end{definition}

The following should be thought of as a generalization of Rasmussen's theorem that characterizes $d_K$ for knots $K$.

\begin{lemma}\label{restrictiondegree}
The operators $\Res_{\mathfrak o_i}$ and $\Res_{\mathfrak o_{ij},\bar{\mathfrak o}_{ij}}$ are both filtered of degree $-2$.
\end{lemma}

\begin{proof}
For $\Res_{\mathfrak o_{ij},\bar{\mathfrak o}_{ij}}$, consider the cobordism formed by first adding a $1$-handle connecting $L_i$ and $L_j$ (in such 
a way that the given relative orientation extends over the cobordism) and then adding a second $1$-handle splitting the resulting component 
back into $L_i\cup L_j$.  Clearly this cobordism induces the map $\Res_{\mathfrak o_{ij},\bar{\mathfrak o}_{ij}}:\mathbb O(L)\to\mathbb O(L)$, and it 
has Euler characteristic $-2$, so we are done.

For $\Res_{\mathfrak o_i}$, let $U$ be the unknot, and consider the map $\mathbb O(L)\to\mathbb O(L)\otimes\mathbb O(U)=\mathbb O(L\sqcup U)$ given 
by $\alpha\mapsto\alpha\otimes\mathfrak o$.  This is filtered of degree $-1$.  Now compose with the map $\mathbb O(L\sqcup U)\to\mathbb O(L)$ given 
by the cobordism obtained by adding a $1$-handle to merge the unknot and $L_i$ (such that the orientation $\mathfrak o$ and the desired 
orientation $\mathfrak o_i$ extend over the $1$-handle).  This cobordism has Euler characteristic $-1$, so the composition 
$\mathbb O(L)\to\mathbb O(L)\otimes\mathbb O(U)\to\mathbb O(L)$ is filtered of degree $-2$.  This map also clearly equals $\Res_{\mathfrak o_i}$.
\end{proof}

\begin{lemma}\label{signfliplemma}
Let $\psi:\mathbb O(L)\to\mathbb O(L)$ be defined by $\psi(\mathfrak o)=\sigma(\mathfrak o)\mathfrak o$, 
where $\sigma(\mathfrak o)=\pm1$ depending on the orientation of some specific component $L_i\subseteq L$.  
Then $\psi$ is filtered of degree $-2$.
\end{lemma}

\begin{proof}
Such a map is a linear combination of $\Res_{\mathfrak o_i}$ and $\Res_{\bar{\mathfrak o_i}}$.
\end{proof}

\begin{lemma}
For every $\alpha\in\mathbb O(L)$, we have $s(\alpha)=\min(s(\alpha+\bar\alpha),s(\alpha-\bar\alpha))$.  In particular, it follows that $s(\alpha)=s(\bar\alpha)$.
\end{lemma}

\begin{proof}
Note that $\alpha+\bar\alpha$ and $\alpha-\bar\alpha$ are in different mod $4$ gradings, so they are sent to different mod $4$ gradings in $\KhL^\ast(L)$.  Thus by Theorem \ref{khleedef} property \ref{smodfour}, we know that $s(\alpha+\bar\alpha)$ and $s(\alpha-\bar\alpha)$ are different mod $4$.  Thus $s(\alpha+\bar\alpha)\ne s(\alpha-\bar\alpha)$, so $s(\alpha)=\min(s(\alpha+\bar\alpha),s(\alpha-\bar\alpha))$.
\end{proof}

Suppose we have a relative orientation $(\mathfrak o,\bar{\mathfrak o})$ of $L$.  Let 
$V_{\mathfrak o,\bar{\mathfrak o}}\subseteq\mathbb O(L)$ be the subspace generated by 
$\mathfrak o$ and $\bar{\mathfrak o}$.  Then let us consider the restriction of the $s$-filtration to 
$V_{\mathfrak o,\bar{\mathfrak o}}=\mathbb Q(\mathfrak o+\bar{\mathfrak o})\oplus\mathbb Q(\mathfrak o-\bar{\mathfrak o})$.  
Note that this direct sum decomposition is into mod $4$ graded pieces; thus the elements $\mathfrak o+\bar{\mathfrak o}$ 
and $\mathfrak o-\bar{\mathfrak o}$ are sent to different mod $4$ gradings in $\KhL^\ast(L)$ which differ 
by exactly $2$.  Thus the $s$-filtration on $V_{\mathfrak o,\bar{\mathfrak o}}$ is completely 
described by the two integers $s(\mathfrak o+\bar{\mathfrak o})$ and $s(\mathfrak o-\bar{\mathfrak o})$ 
(which differ by $2$ mod $4$).  By Lemma \ref{signfliplemma}, $s(\mathfrak o+\bar{\mathfrak o})$ and 
$s(\mathfrak o-\bar{\mathfrak o})$ differ by exactly two, and we let the oriented 
$s(L,\mathfrak o)=\frac 12[s(\mathfrak o+\bar{\mathfrak o})+s(\mathfrak o-\bar{\mathfrak o})]$.

\begin{definition}\label{beliakovawehrliinvariant}
The invariant constructed in the previous paragraph is $s(L,\mathfrak o)$.  It was first defined by Beliakova and Wehrli \cite{wehrli}.
\end{definition}

For a knot $K$, there is just one relative orientation.  This gives Rasmussen's invariant $s(K)$, which determines the 
$s$-filtration on $\KhL^\ast(K)$.  For links, however, there is much more to 
the $s$-filtration on $\KhL^\ast(L)$ that is not captured by the function $\mathfrak o\mapsto s(L,\mathfrak o)$.  For example, for any 
alternating link $L$ with zero linking matrix, all $s(L,\mathfrak o)$ are equal (say, to $s_0$), and 
$\sum_{i,j}d_L(i,j)t^iq^j=2^{\left|L\right|-1}q^{s_0}(q+q^{-1})$.  On the other hand, if $L$ is unlink on 
$n$ components, then $s(L,\mathfrak o)$ are all equal (this time to $1-n$), however in this case $\sum_{i,j}d_L(i,j)t^iq^j=(q+q^{-1})^n$.  
Thus for link concordance, the $s$-filtration on $\mathbb O(L)$ is a stronger invariant than the function $\mathfrak o\mapsto s(L,\mathfrak o)$.

\section{Examples}\label{examples}

We now summarize some calculations of the invariant $d_L:\mathbb Z\times\mathbb Z\to\mathbb Z_{\geq 0}$ for some links $L$.  
We used the package \texttt{KnotTheory`} maintained by Bar-Natan \cite{knottheory}, in particular the program to calculate 
Khovanov homology written by Scott Morrison.  This allows us to calculate $\Kh^\ast(L)$ for the link in question.  
We use the simple fact that if $\dim\Kh^h(L)=\dim\KhL^h(L)$, then by virtue of the spectral sequence from $\Kh^\ast(L)$ 
to $\KhL^\ast(L)$, the support of the $s$-filtration on $\KhL^h(L)$ is given exactly by the $q$-graded dimension of $\Kh^h(L)$.  
Many interesting links have lots of crossings, and thus computing the Khovanov homology is time consuming on a computer; 
we just list the cases that we have been able to compute.

Most of the links in the standard link tables are quasi-alternating, so they do not present a particularly interesting 
case for the filtration on $\KhL^\ast(L)$ (it is just supported in two 
levels, so only their absolute height is interesting).  So instead, we've taken as our examples some links with extra structure.

The function $d_L$ is a link concordance invariant, and thus there are some easy corollaries using Theorem \ref{mainthm2} 
distinguishing the link concordance classes of the links we consider below from other links whose $d_L$ one could calculate 
(e.g.\ one can easily see which are concordant to a quasi-alternating link).  Theorem \ref{mainthm2} also implies effective bounds on the genus 
of component-preserving orientable cobordisms between these links and links with certain splitting numbers.

\subsection{Cablings of $T_{2,p}$}

Let $L_p$ be the $(2,0)$-cabling of $T_{2,p}$.  Then the linking matrix of $L_p$ is zero, and we have (for $p$ odd, $1\leq p\leq 11$):
\begin{equation}\label{cableformula}
\sum_{i,j}d_{L_p}(i,j)\cdot t^iq^j=1+q^2+q^{2p-4}+q^{2p-2}
\end{equation}
We conjecture that this is true for all odd $p\geq 1$.  We can prove the following:

\begin{lemma}
Fix an odd $p\geq 5$.  Let $(\mathfrak o_+,\bar{\mathfrak o}_+)$ be the relative orientation of $L_p$ where the two strands are oriented in the same direction, and let $V_+$ be the subspace of $\mathbb O(L_p)$ generated by $(\mathfrak o_+,\bar{\mathfrak o}_+)$.  Similarly define $(\mathfrak o_-,\bar{\mathfrak o}_-)$ and $V_-$ with the two strands oriented in opposite directions.

Then the $s$-filtration on $\mathbb O(L)=V_+\oplus V_-$ splits up as a filtration on each $V_\pm$.  Furthermore, $V_+$ is supported in filtration degrees $(2p-4,2p-2)$, and $V_-$ is supported in degrees $(0,2)$ or $(-2,0)$.  In particular, we have:
\begin{equation}\label{weakenedcableformula}
\sum_{i,j}d_{L_p}(i,j)\cdot t^iq^j=\left\{\begin{matrix}1+q^2\cr\text{or}\cr q^{-2}+1\end{matrix}\right\}+q^{2p-4}+q^{2p-2}
\end{equation}
\end{lemma}

Of course, it is not true in general that the $s$-filtration splits up as a direct sum over all relative orientations $(\mathfrak o,\bar{\mathfrak o})$ of filtrations on $V_{\mathfrak o,\bar{\mathfrak o}}$ (for example, this fails for any split link by Theorem \ref{khleedef} property \ref{tensorproperty}).

\begin{proof}
We thank the referee for the argument in this paragraph.  The standard diagram for $T_{2,p}$ has $p$ positive crossings, so the $2$-cabling in the blackboard framing is the $(2,2p)$-cable, which we call $L_p'$.  Orient both $L_p$ and $L_p'$ via $\mathfrak o_+$ (we let $\mathfrak o_+$ denote the orientation on $L_p'$ corresponding naturally to $\mathfrak o_+$ on $L_p$).  Now $L_p'$ has a positive diagram coming from the positive diagram of $T_{2,p}$.  In this diagram, there are $4p$ crossings and $4$ circles in the oriented resolution.  Thus by Rasmussen \cite[p439, section 5.2]{rasmussen}, we have $s_{L_p'}(\mathfrak o_+)=4p-4$.  Now we can transform $L_p'$ into $L_p$ by $p$ crossing changes (they differ by $2p$ half-twists between the two strands).  Thus $p$ iterated applications of Lemma \ref{crossingchangelemma} imply that:
\begin{equation}
s_{L_p}(\mathfrak o_+)\geq 2p-4
\end{equation}
On the other hand, $L_p$ bounds two parallel copies of a Seifert surface for $T_{2,p}$, each of genus $\frac 12(p-1)$.  These give a component-preserving cobordism of genus $p-1$ from $L_p$ to the unlink.  Every orientation $\mathfrak o$ of the unlink has $s(\mathfrak o)=-2$.  Thus we have:
\begin{equation}
s_{L_p}(\mathfrak o_+)\leq s_{\text{unlink}}(\mathfrak o)+2(p-1)=2p-4
\end{equation}
Thus $s_{L_p}(\mathfrak o_+)=2p-4$.  By the discussion surrounding Definition \ref{beliakovawehrliinvariant}, this shows that the restriction of the $s$-filtration to $V_+$ is supported in degrees $(2p-4,2p-2)$.

Now for $\mathfrak o_-$, observe that adding a $1$-handle merging the two components of $L_p$ yields the unknot $U$.  Thus if we orient $L_p$ by $\mathfrak o_-$, we have maps $\mathbb O(L_p)\to\mathbb O(U)\to\mathbb O(L_p)$, and in fact they give \emph{isomorphisms}:
\begin{equation}
V_-\xrightarrow\sim\mathbb O(U)\xrightarrow\sim V_-
\end{equation}
which are filtered of degree $-1$.  Since $\mathbb O(U)$ is supported in degrees $\pm 1$, and $V_-$ is supported in even degrees (which differ by exactly two), we see that the support of $V_-$ is either $(0,2)$ or $(-2,0)$.

Now it remains to show that the $s$-filtration on $\mathbb O(L)$ decomposes as the direct sum of the filtrations on each $V_\pm$.  In other words, we need to show that $s(\alpha_++\alpha_-)=\min(s(\alpha_+),s(\alpha_-))$ for all pairs $\alpha_\pm\in V_\pm$.  Our assumption $p\geq 5$ implies $2p-4>2$, so by the results above, we have $s(\alpha_+)\ne s(\alpha_-)$ (unless $\alpha_+=\alpha_-=0$).  It follows that $s(\alpha_++\alpha_-)=\min(s(\alpha_+),s(\alpha_-))$.
\end{proof}

\subsection{$T_{n,n}$}

We now consider the $(n,n)$-torus links for $1\leq n\leq 6$ (with all components oriented the same direction).  We have:
\begin{align}
\sum_{i,j}d_{T_{1,1}}(i,j)\cdot t^iq^j&=q+q^{-1}\cr
\sum_{i,j}d_{T_{2,2}}(i,j)\cdot t^iq^j&=[1+q^2]+(tq^3)^2[q^{-2}+1]\cr
\sum_{i,j}d_{T_{3,3}}(i,j)\cdot t^iq^j&=[q^3+q^5]+(tq^3)^4[q^{-3}+3q^{-1}+2q]\cr
\sum_{i,j}d_{T_{4,4}}(i,j)\cdot t^iq^j&=[q^8+q^{10}]+(tq^3)^6[q^{-2}+4+3q^2]+(tq^3)^8[q^{-4}+3q^{-2}+2]\cr
\sum_{i,j}d_{T_{5,5}}(i,j)\cdot t^iq^j&=[q^{15}+q^{17}]+(tq^3)^8[q+5q^3+4q^5]+(tq^3)^{12}[q^{-5}+5q^{-3}+9q^{-1}+5q]\cr
\sum_{i,j}d_{T_{6.6}}(i,j)\cdot t^iq^j&=[q^{24}+q^{26}]+(tq^3)^{10}[q^6+6q^8+5q^{10}]\cr&+(tq^3)^{16}[q^{-4}+6q^{-2}+14+9q^2]+(tq^3)^{18}[q^{-6}+5q^{-4}+9q^{-2}+5]
\end{align}
In accordance with Theorem \ref{mainthm2} property \ref{orchange}, it is natural to separate out factors of $tq^3$.  
For $n=5,6$, computing the final answer requires use of Theorem \ref{mainthm2} property \ref{supportmod}, in particular 
the fact that the dimensions of $\KhL^\ast$ supported in $s$-filtration level $\left|L\right|$ and $\left|L\right|+2$ are equal
(this enables us to see which parts of $\Kh^\ast$ are killed in the spectral sequence).

The referee has noticed the following pattern for the values of $d_{T_{n,n}}$.  Define polynomials 
$P_{n,k}\in\mathbb Q[q,q^{-1}]$ for $n\geq 0$ and $0\leq 2k\leq n$ by the recurrence:
\begin{align}
P_{0,0}&=1\\
\text{for $n\geq 1$}\quad P_{n,k}&=\begin{cases}
qP_{n-1,k-1}+q^{-1}P_{n-1,k}&2k+2\leq n\cr
qP_{n-1,k-1}+\frac 12(1+q^{-1})P_{n-1,k}&2k+1=n\cr
2P_{n-1,k-1}&2k=n
\end{cases}
\end{align}
where we interpret $P_{n,k}$ as zero if $k<0$.  Certainly $P_{n,k}$ are some sort of $q$-deformed binomial coefficients.  Then for $1\leq p\leq 6$, we have:
\begin{equation}\label{pattern}
\sum_{i,j}d_{T_{n,n}}(i,j)\cdot t^iq^j=\sum_{k=0}^n(tq^3)^{2k(n-k)}q^{(n-2k)^2}P_{n,\min(k,n-k)}(q^2)
\end{equation}
One would naturally conjecture that this equality holds for all larger $p$ as well.

\bibliographystyle{plain}
\bibliography{leeconcordance}

\begin{thebibliography}{10}

\bibitem{barnatanexpo}
Dror Bar-Natan.
\newblock On {K}hovanov's categorification of the {J}ones polynomial.
\newblock {\em Algebr. Geom. Topol.}, 2:337--370 (electronic), 2002.

\bibitem{barnatan}
Dror Bar-Natan.
\newblock Khovanov's homology for tangles and cobordisms.
\newblock {\em Geom. Topol.}, 9:1443--1499, 2005.

\bibitem{bnalg}
Dror Bar-Natan.
\newblock Fast {K}hovanov homology computations.
\newblock {\em J. Knot Theory Ramifications}, 16(3):243--255, 2007.

\bibitem{knottheory}
Dror Bar-Natan and Scott Morrison.
\newblock \texttt{KnotTheory`} {P}ackage.
\newblock {\em
  \texttt{http://katlas.org/\linebreak[0]wiki/\linebreak[0]The\_Mathematica\_P%
ackage\_KnotTheory`}}, 2008.

\bibitem{wehrli}
Anna Beliakova and Stephan Wehrli.
\newblock Categorification of the colored {J}ones polynomial and {R}asmussen
  invariant of links.
\newblock {\em Canad. J. Math.}, 60(6):1240--1266, 2008.

\bibitem{caprau}
Carmen~Livia Caprau.
\newblock {$\rm sl(2)$} tangle homology with a parameter and singular
  cobordisms.
\newblock {\em Algebr. Geom. Topol.}, 8(2):729--756, 2008.

\bibitem{fixingfunctoriality}
David Clark, Scott Morrison, and Kevin Walker.
\newblock Fixing the functoriality of {K}hovanov homology.
\newblock {\em Geom. Topol.}, 13(3):1499--1582, 2009.

\bibitem{spc4calc}
Michael Freedman, Robert Gompf, Scott Morrison, and Kevin Walker.
\newblock Man and machine thinking about the smooth 4-dimensional {P}oincar\'e
  conjecture.
\newblock {\em Quantum Topol.}, 1(2):171--208, 2010.

\bibitem{powell}
Stefan Friedl and Mark Powell.
\newblock Cobordisms to weakly splittable links.
\newblock {\em \texttt{arXiv math.GT 1112.3685}}, 2011.

\bibitem{greene}
Joshua Greene.
\newblock Homologically thin, non-quasi-alternating links.
\newblock {\em Math. Res. Lett.}, 17(1):39--49, 2010.

\bibitem{hatcher}
Allen~E. Hatcher.
\newblock A proof of the {S}male conjecture, {${\rm Diff}(S^{3})\simeq {\rm
  O}(4)$}.
\newblock {\em Ann. of Math. (2)}, 117(3):553--607, 1983.

\bibitem{jacob}
Magnus Jacobsson.
\newblock An invariant of link cobordisms from {K}hovanov homology.
\newblock {\em Algebr. Geom. Topol.}, 4:1211--1251 (electronic), 2004.

\bibitem{alternation}
Akio Kawauchi.
\newblock On alternation numbers of links.
\newblock {\em Topology Appl.}, 157(1):274--279, 2010.

\bibitem{khovanov}
Mikhail Khovanov.
\newblock A categorification of the {J}ones polynomial.
\newblock {\em Duke Math. J.}, 101(3):359--426, 2000.

\bibitem{kmmilnord1}
P.~B. Kronheimer and T.~S. Mrowka.
\newblock Gauge theory for embedded surfaces. {I}.
\newblock {\em Topology}, 32(4):773--826, 1993.

\bibitem{kmmilnorsw}
P.~B. Kronheimer and T.~S. Mrowka.
\newblock The genus of embedded surfaces in the projective plane.
\newblock {\em Math. Res. Lett.}, 1(6):797--808, 1994.

\bibitem{kmmilnord2}
P.~B. Kronheimer and T.~S. Mrowka.
\newblock Gauge theory for embedded surfaces. {II}.
\newblock {\em Topology}, 34(1):37--97, 1995.

\bibitem{lee}
Eun~Soo Lee.
\newblock An endomorphism of the {K}hovanov invariant.
\newblock {\em Adv. Math.}, 197(2):554--586, 2005.

\bibitem{livingston}
Charles Livingston.
\newblock Computations of the {O}zsv\'ath-{S}zab\'o knot concordance invariant.
\newblock {\em Geom. Topol.}, 8:735--742 (electronic), 2004.

\bibitem{quasithin}
Ciprian Manolescu and Peter Ozsv{\'a}th.
\newblock On the {K}hovanov and knot {F}loer homologies of quasi-alternating
  links.
\newblock In {\em Proceedings of {G}\"okova {G}eometry-{T}opology {C}onference
  2007}, pages 60--81. G\"okova Geometry/Topology Conference (GGT), G\"okova,
  2008.

\bibitem{signaturemurasugi}
Kunio Murasugi.
\newblock On a certain numerical invariant of link types.
\newblock {\em Trans. Amer. Math. Soc.}, 117:387--422, 1965.

\bibitem{oszhfk}
Peter Ozsv{\'a}th and Zolt{\'a}n Szab{\'o}.
\newblock Holomorphic disks and knot invariants.
\newblock {\em Adv. Math.}, 186(1):58--116, 2004.

\bibitem{oszquasialter}
Peter Ozsv{\'a}th and Zolt{\'a}n Szab{\'o}.
\newblock On the {H}eegaard {F}loer homology of branched double-covers.
\newblock {\em Adv. Math.}, 194(1):1--33, 2005.

\bibitem{oszhfl}
Peter Ozsv{\'a}th and Zolt{\'a}n Szab{\'o}.
\newblock Holomorphic disks, link invariants and the multi-variable {A}lexander
  polynomial.
\newblock {\em Algebr. Geom. Topol.}, 8(2):615--692, 2008.

\bibitem{rasrefined}
Jacob Rasmussen.
\newblock Khovanov's invariant for closed surfaces.
\newblock {\em \texttt{arXiv math.GT 0502527}}, 2005.

\bibitem{rasmussen}
Jacob Rasmussen.
\newblock Khovanov homology and the slice genus.
\newblock {\em Invent. Math.}, 182(2):419--447, 2010.

\bibitem{kwalker}
Kevin Walker\phantom{x}(mathoverflow.net/users/284).
\newblock Is the complete functorial structure for {K}hovanov--{L}ee homology
  known?
\newblock MathOverflow.
\newblock http://mathoverflow.net/questions/70902 (version: 2011-07-22).

\end{thebibliography}

\end{document}